\documentclass[a4paper,12pt]{article}
\usepackage{amssymb, amsmath, afterpage}
\usepackage{amsfonts}
\usepackage[dvips]{graphics}
\newcommand{\mybox}{\hfill $\square$}
\newcommand{\+}[1]{\ensuremath{\overset{+}{#1}}}
\newcommand{\pli}[1]{\ensuremath{\overset{+}{\imath}}}
\newcommand{\mi}[1]{\ensuremath{\overset{-}{#1}}}

\newcommand{\sym}{\ensuremath{\mathrm{Sym}}}

\newcommand{\qed}{\hfill $\square$ \\}

\newcommand{\zz}{\mathbb{Z}}

\newtheorem{thm}{Theorem}[section]
\newtheorem{lemma}[thm]{Lemma}         

\newtheorem{conj}[thm]{Conjecture}
\newtheorem{defn}[thm]{Definition}
\newtheorem{cor}[thm]{Corollary}
\newenvironment{proof}{\paragraph{Proof}}{\mybox}
\begin{document}
    \title{\Large\textbf{Involution Statistics in Finite Coxeter Groups}}
\date{} 
      \author{S.B. Hart and P.J. Rowley
\thanks{The authors wish to acknowledge support from
       a London Mathematical Society `Research in Pairs' grant and
       from the Department of Economics, Mathematics and Statistics at Birkbeck.
      }}
  \maketitle
\vspace*{-10mm}

\abstract{Let $W$ be a finite Coxeter group and $X$ a subset of $W$.
The length polynomial $L_{W,X}(t)$ is defined by $L_{W,X}(t) =
\sum_{x \in X} t^{\ell(x)}$, where $\ell$ is the length function on
$W$. In this article we derive expressions for the length polynomial
where $X$ is any conjugacy class of involutions, or the set of all
involutions, in any finite Coxeter group $W$. In particular, these
results correct errors in \cite{dukes} for the involution length
polynomials of Coxeter groups of type $B_n$ and $D_n$. Moreover, we
give a counterexample to a unimodality conjecture stated in
\cite{dukes}.}

\section{Introduction}

The purpose of this article is to derive statistics about the
distribution of lengths of involutions in Coxeter groups. The {\em
length polynomial}, $L_{W,X}(t)$, where $W$ is a finite Coxeter
group and $X \subseteq W$, is the principal object of study here. It
is defined by $$L_{W,X}(t) = \sum_{x\in X} t^{\ell(x)},$$ where
$\ell$ is the length function on $W$. If $X = W$, this is the
well-known Poincar\'{e} polynomial (see 1.11 of \cite{humphreys}).
For the special case where $X = \{x \in W: x^2 =1\}$ is the set of
involutions in $W$ together with the identity element, we write
$L_W(t)$ for the polynomial
$$L_W(t) = L_{W,X}(t) = \sum_{x \in X} t^{\ell(x)}.$$ We refer to
$L_W(t)$ as the
{\em involution length polynomial} of $W$.\\

 In this
article we obtain expressions for the length polynomials of all
conjugacy classes of involutions in finite Coxeter groups (and hence
for the sets of all involutions in these groups). For type $A$ this
is known \cite{desarmenien}, but we could only find statements, not
proofs, in the literature, so we have included a proof here. In
\cite{dukes} expressions for $L_{W(B_n)}(t)$ and $L_{W(D_n)}(t)$ are
given, but unfortunately the proofs contain errors which lead to the
results being incorrect. Finally we remark that the length
polynomial for the special case where $X$ is the set of reflections
in a Coxeter group $W$ has been studied in another guise
\cite{deman}. Theorem 4.1 of \cite{deman} gives the generating
function for counting the depth of roots of a Coxeter system of
finite rank. The number of reflections of length $\ell$ in a Coxeter
group equals the number of positive roots of depth
$\frac{\ell+1}{2}$, and so the polynomial $L_{W,X}(t)$, where $X$ is
the set of reflections, may be obtained.
\\

There are two main motivations for this work. The first is that
there are known results counting involutions in the symmetric group
that have a given number of inversions. Of course the symmetric
group is a Coxeter group of type $A$, and it is well known that the
number of inversions of an element is equal to its length, in the
Coxeter group context. Thus it is natural to wonder about
generalisations of results on inversion numbers in symmetric groups
to results relating to lengths in Coxeter groups.\\

The second motivation is that Coxeter groups have a very special
relationship with involutions. They are generated by involutions
(known as fundamental reflections). The set of reflections of a
Coxeter group (conjugates of fundamental reflections) is in
one-to-one correspondence with its set of positive roots. For a
Coxeter group of rank $n$, every involution can be expressed as a
product of at most $n$ orthogonal reflections. Every element of a
finite Coxeter group can be expressed as a product of at most two
involutions \cite{carter}. Moreover, the conjugacy classes of
involutions have a particularly nice structure. Due to a result of
Richardson \cite{richardson}, it is straightforward to determine
them from the Coxeter graph. The involutions of minimal and maximal
length in a conjugacy class are well understood \cite{prinv2}, and
the lengths of involutions behave well with respect to conjugation,
in that if $x$ is an involution and $r$ is a fundamental reflection,
then either $\ell(rxr) = \ell(x) + 2$, or $\ell(rxr) = \ell(x) - 2$,
or $rxr = x$. The involutions of minimal length in a conjugacy class
are central elements of parabolic subgroups, and so are fairly
easily counted. So again it is natural to ask what can be determined
about the length distribution of involutions, both in a conjugacy
class and in
a Coxeter group as a whole.\\

We may now state our results. Let $n, m$ and $\ell$ be integers. We
define $\alpha_{n,m,\ell}$ to be the number of involutions in
$W(A_{n-1}) \cong \sym(n)$ of length $\ell$, having $m$
transpositions. We adopt the convention that the identity element is
an involution, so that we have $\alpha_{n,0,0} = 1$. Note that
$\alpha_{n,0,\ell} = 0$ for all $\ell \neq 0$ and
$\alpha_{n,m,\ell} = 0$ for any $\ell < 0$.\\

We define the following polynomial
$$L_{n,m}(t) = \sum_{\ell=0}^{\infty} \alpha_{n,m,\ell}t^\ell.$$

This is essentially a shorthand for $L_{W(A_{n-1}),X}(t)$ where $X$
is the set of involutions (including the identity) whose expression
as a product of disjoint cycles contains precisely $m$
transpositions.

\begin{thm}\label{sn}
Let $n \geq 3$ and $m \geq 1$. If $\ell < m$ then $\alpha_{n,m,\ell}
= 0$. If $\ell \geq m$ then
 $$\alpha_{n,m,\ell} =  \alpha_{n-1,m,\ell} + \sum_{k=1}^{n-1} \alpha_{n-2,m-1,\ell + 1 -
 2k}.$$
Moreover $L_{n,0}(t) = 1$, $L_{1,1}(t) = 0$, $L_{2,1}(t) = t$,
$L_{n,m}(t) = 0$ for $n < 2m$, and for all other positive integers
$n, m$,
$$L_{n,m}(t) = L_{n-1,m}(t) +
\frac{t(t^{2n-2}-1)}{t^2-1}L_{n-2,m-1}(t).$$
\end{thm}

For example, to find the number of double transpositions of length 6
in $W(A_{4}) \cong \sym(5)$, we calculate
\begin{eqnarray*}\alpha_{5,2,6} &=& \alpha_{4,2,6} + [\alpha_{3,1,5}
 + \alpha_{3,1,3} + \alpha_{3,1,1} + \alpha_{3,1,-1}] \\ &=& 1 + [0 + 1 + 2 + 0] = 4.\end{eqnarray*}

Note that \cite{desarmenien}, for example, defines a related
polynomial $I_n(x,q)$ in two variables, $x$ and $q$, where the
coefficient of $q^jx^k$ is the number of involutions in $\sym_n$
with $k$ fixed points and $j$ inversions (that is, length $j$). He
states ``on v\'{e}rifie ais\'{e}ment que ces polyn\^{o}mes satisfont
la r\'{e}currence $I_0(x,q) = 1$, $I_1(x,q) = x$, $I_{n+1}(x,q) =
xI_n(x,q) + \frac{1-q^{2n}}{1-q^2}qI_{n-1}(x,q)$, $n\geq 1$''. Of
course one may derive Theorem \ref{sn} from this statement, but we
felt it may be helpful to include a direct proof here, especially as
our result for $W(A_{n-1}) \cong \sym(n)$ is arrived at as a subcase
of a
general argument that also covers types $B_n$ and $D_n$.\\

 Our
results for types $B_n$ and $D_n$ are as follows. (A detailed
description of these groups as groups of permutations will be given
in Section 2.) Let $\beta_{n,m,e,\ell}$ be the number of involutions
of length $\ell$ in $W(B_n)$ whose expression as a product of
disjoint signed cycles contains $m$ transpositions and $e$ negative
1-cycles. Also let $\delta_{n,m,e,\ell}$ be the number of
involutions of length $\ell$ in $W(D_n)$ whose expression as a
product of disjoint signed cycles contains $m$ transpositions and
$e$ negative 1-cycles.
 Note that $\beta_{n,0,0,0} = \delta_{n,0,0,0} = 1$, as we include the identity element in our count. Let $$L_{n,m,e}(t) =
 \left\{ \begin{array}{ll} \sum_{\ell=0}^{\infty} \beta_{n,m,e,l}t^{\ell}
 & {\mbox{ when $W = W(B_n)$}};\\
 \sum_{\ell=0}^{\infty}
\delta_{n,m,e,\ell}t^{\ell}
 & {\mbox{ when $W = W(D_n)$}}.\end{array}\right.$$

Again, $L_{n,m,e}(t)$ is another way of writing $L_{W,X}(t)$ where
$X$ is the appropriate involution class of $W$, for $W$ of type
$B_n$ or $D_n$.\\

 It is not common to work with
the groups of type $B_1, D_1$, $D_2$ and $D_3$, because they are
isomorphic to certain groups of type $A$ (or in the case of $D_2$,
to $A_1 \times A_1$), but there is a natural definition in line with
the usual definitions (see Section 2) of $W(B_n)$ and $W(D_n)$ (for
example $W(D_3)$ is the group of positive elements of $W(B_3)$).
Therefore, for the purposes of recursion, we will work with
$L_{n,m,e}(t)$ for $W$ of types $B_1, D_1, D_2$, and $D_3$. We have
the following two theorems.

\begin{thm}\label{bn}
Suppose $W = W(B_n)$, and let $n \geq 3$, $m \geq 0$ and $e \geq 0$.
If $\ell < m + e$ then $\beta_{n,m,e,\ell} = 0$. If $\ell \geq m +
e$ then
$$\beta_{n, m,e, \ell} = \beta_{n-1,m,e,\ell} + \beta_{n-1,m,e-1,\ell+ 1-2n} +
\sum_{k=1}^{2n-2} \beta_{n-2,m-1,e,\ell + 1 - 2k}.$$

Moreover $L_{n,0,0}(t) = 1$ for all positive integers $n$,
$L_{1,0,1}(t) = t$, $L_{2,1,0}(t) = t + t^3$, $L_{2,0,1}(t) = t +
t^3$, $L_{2,0,2}(t) = t^4$, $L_{n,m,e}(t) = 0$ whenever $n < 2m +
e$, and for all $n \geq 3$ and $m, e \geq 0$,
$$L_{n,m,e}(t) = L_{n-1, m,e}(t) +
t^{2n-1}L_{n-1,m,e-1}(t) + \frac{t(t^{4n-4}-1)}{t^2-1}
L_{n-2,m-1,e}(t).$$
\end{thm}

Before we can state the next theorem we define a further polynomial
$D_{n,m,e}(t)$, where $n,m$ and $e$ are non-negative integers. Set
$D_{n,0,0}(t) = 1$ and if $2m + e > n$, define $D_{n,m,e}(t) = 0$.
Set $D_{1,0,1}(t) = t$, $D_{2,1,0}(t) = 2t$, $D_{2,0,1}(t) =  1 +
t^2$, $D_{2,0,2}(t) = t^2$, and for all $n \geq 3$,

$$D_{n,m,e}(t) = D_{n-1, m,e}(t) + t^{2n-2}D_{n-1,m,e-1}(t) +
\frac{t(1+t^{2n-4})(t^{2n-2}-1)}{t^2-1} D_{n-2,m-1,e}(t). $$

\begin{thm}\label{dn} Suppose $W = W(D_n)$ and let $\ell$ be a non-negative integer.
Then $\delta_{n,m,e,\ell}$ is the coefficient of $t^{\ell}$ in
$D_{n,m,e}(t)$ when $e$ is even, and zero otherwise. Furthermore
$L_{n,m,e}(t) = D_{n,m,e}(t)$ when $e$ is even and $L_{n,m,e,}(t) =
0$ when $e$ is odd.
\end{thm}

By summing the involutions of various types of a given length, we
can obtain the polynomials $f_W(t)$ for $W$ of types $A$, $B$ and
$D$. This result is Corollary \ref{all}. We postpone the statement
and proof until Section 4 because the $D_n$ case requires extra
notation.\\

This article is structured as follows: in Section 2 we give the
basic facts about presentations and root systems of the classical
Weyl groups, and describe the relationship between roots and length.
In Section 3 we prove the general results for classical Weyl groups,
which are then applied in Section 4 to each of types $A$, $B$ and
$D$ in turn. In Section 5 we give the length polynomials for
involution conjugacy classes in the exceptional finite Coxeter
groups. Finally in Section 6 we discuss briefly some conjectures
about unimodality of the sequences of coefficients of the
polynomials $L_{W,X}(t)$, where $X$ is either a conjugacy class of
involutions, or the set of involutions of even length, or the set of
involutions of odd length, in a finite Coxeter group $W$.

\section{Root Systems and Length}

Suppose $W$ is a Coxeter group, and $R$ its set of fundamental
reflections. Then $W = \langle R \rangle$. For any $w \in W$, the
{\em length} of $w$, denoted $\ell(w)$ is the length of any shortest
expression for $w$ as a product of fundamental reflections. That is
$$\ell(w) = \min\{k \in \zz: w = r_1r_2\cdots r_k, {\mbox{
some $r_1, \ldots, r_k \in R$}}\}.$$

By convention the identity element has length zero. If the choice of
Coxeter group $W$ is not clear from context, we will write either
$\ell_W$ or sometimes $\ell_{Y}$ if $W$ is of type $Y$. The length
function has been extensively studied, not least because of its link
to the root system, which we will now describe.\\

To every Coxeter group $W$ we may assign a root system $\Phi$, a set
of positive roots $\Phi^+$ and a set of negative roots
$\Phi^-:=-\Phi^+$, such that $\Phi$ is the disjoint union of the
positive and negative roots. Frequently we shall write $\rho > 0$ to
mean $\rho \in \Phi^+$ and $\rho < 0$ to mean $\rho \in \Phi^-$. The
Coxeter group $W$ acts faithfully on its root system. For $w \in W$,
define
$$N(w) = \{\alpha \in \Phi^+: w\cdot \alpha \in \Phi^-\}.$$

One of many connections between $\Phi$ and the length function
$\ell$ is the fact that $\ell(w) = |N(w)|$ (see, for example,
Section 5.6 of \cite{humphreys}).\\

Every finite Coxeter group is a direct product of irreducible finite
Coxeter groups. The finite irreducible Coxeter groups were
classified by Coxeter \cite{coxeter} (see also \cite{humphreys}).

\begin{thm}\label{classification} An irreducible finite Coxeter group is either of type
$A_n (n\geq 1)$, $B_n (n \geq 2)$, $D_n (n \geq 4)$, $E_6$, $E_7$,
$E_8$, $F_4$, $H_3$, $H_4$ or $I_n$ ($n \geq 5$).
\end{thm}

The next result (the proof of which is straightforward) shows that
for questions about conjugacy classes of involutions, or the set of
all involutions, or indeed the set of involutions of odd or even
length, in a finite Coxeter group, it is sufficient to analyse the
irreducible cases only.

\begin{lemma} Let $W$ be a finite Coxeter group, and $X$ a union of involution conjugacy classes
of $W$. Suppose that $W$ is isomorphic to the direct product $W_1
\times W_2 \times \cdots \times W_k$, where each $W_i$ is an
irreducible Coxeter group. Then $X$ is the direct product of sets
$X_1$, $X_2$, \ldots, $X_k$. Each $X_i$ is either $\{1\}$ or is a
union of involution conjugacy classes of $W_i$, and $L_{W,X}(t) =
\prod_{i=1}^k L_{W_i,X_i}(t)$.\end{lemma}

We next discuss concrete descriptions of the Coxeter groups of types
$A_n, B_n$ and $D_n$ which will feature in a number of our proofs.
First, $W(A_n)$ may be viewed as being $\sym(n+1)$. The elements of
$W(B_n)$ can be thought of as signed permutations of $\sym(n)$. For
example, the element $w = (\+ 1 \mi 2)$ of $W(B_n)$ is given by
$w(1) = 2$, $w(-1) = -2$, $w(2) = -1$ and $w(-2) = 1$.  We say a
cycle in an element of $W(B_n)$ is of negative sign type if it has
an odd number of minus signs, and positive sign type otherwise. An
element $w$ expressed as a product $g_1g_2\cdots g_k$ of disjoint
signed cycles is {\em positive} if the product of all the sign types
of the cycles is positive, and negative otherwise. The group
$W(D_n)$ consists of all positive elements of $W(B_n)$, while the
group $W(A_{n-1})$ consists of all elements of $W(B_n)$ whose cycles
contain only plus signs. Even if $w$ is positive, it may contain
negative cycles, which we wish on occasion to consider separately,
so when considering elements of $W(D_n)$ we often work in the
environment of $W(B_n)$ to avoid ending up with non-group elements.\\

Let $V$ be an $n$-dimensional real vector space with orthonormal
basis $\{e_1, \ldots, e_n\}$. Then we may take the root system $B_n$
to have positive roots of the form $e_j \pm e_i$ for $1 \leq i < j
\leq n$ and $e_i$ for $1\leq i \leq n$. The positive roots of $D_n$
are of the form $e_j \pm e_i$ for $1 \leq i < j \leq n$. The
positive roots of $A_{n-1}$ are of the form $e_j - e_i$ for $1 \leq
i < j \leq n$. Therefore the root system $D_n$ consists of the long
roots of the root system $B_n$, and the root system $A_{n-1}$
is contained in the set of long roots of $B_n$.  \\

From the fact that for an element $w$ of a Coxeter group $W$ we have
$\ell(w) = |N(w)|$, we deduce the following.

\begin{thm}\label{ls} Let $n$ be a positive integer greater than 1, and
let $W$ be of type $A_{n-1}, B_n$ or $D_n$. Set \begin{eqnarray*}
\Lambda &=& \{e_j - e_i, e_j + e_i: 1 \leq i < j \leq n\}{\mbox{ and}}\\
\Sigma &=& \{e_i: 1 \leq i \leq n\}.\end{eqnarray*}%
For $w \in W$, set
\begin{eqnarray*}\Lambda(w) &=& \{\alpha \in \Lambda: w\cdot \alpha < 0\} \;\; {\mbox{ and}}\\
\Sigma(w) &=& \{\alpha \in \Sigma: w \cdot \alpha < 0\}.\end{eqnarray*}%
Then $\ell_{B_n}(w) = |\Lambda(w)| + |\Sigma(w)|$. If $w$ is an
element of $W(D_n)$, then $\ell_{D_n}(w) = |\Lambda(w)|$, and if $w$
is an element of $W(A_{n-1})$, then $\ell_{A_{n-1}}(w) =
|\Lambda(w)|$.\end{thm} %
\begin{proof} Note that $\Lambda(w)$ is the set of long positive
roots taken negative by $w$, and $\Sigma(w)$ is the set of short
positive roots taken negative by $w$. The results for
$\ell_{B_n}(w)$ and $\ell_{D_n}(w)$ are now immediate. For the case
where $w \in W(A_{n-1})$, note that $w\cdot(e_j + e_i) = e_{w(j)} +
e_{w(i)} > 0$ for all $i, j$, so $|\Lambda(w)| = |\{e_j - e_i: 1\leq
i < j < n, w\cdot (e_j - e_i) < 0\}| = \ell_{A_{n-1}}(w)$. So the
result
follows.\end{proof}\\

In the proof of the next lemma we will require the elementary
observation that \linebreak $-h\cdot \Lambda(h) = \Lambda(h^{-1})$,
and hence $|\Lambda(h^{-1})| = |\Lambda(h)|$.

\begin{lemma} \label{ngh} Let $W$ be of type $A_{n-1}$, $B_n$ or $D_n$. Let $g, h \in
W$. Then $$|\Lambda(gh)| = |\Lambda(g)| + |\Lambda(h)| -
2|\Lambda(g) \cap \Lambda(h^{-1})|.$$
\end{lemma}

\begin{proof} Suppose $\alpha \in \Lambda(gh)$. Then either
$h\cdot\alpha \in \Lambda(g)$ or $\alpha \in \Lambda(h)$ and
$h\cdot\alpha \notin -\Lambda(g)$. That is, $\Lambda(gh)$ is the
disjoint union of $h^{-1}\cdot(\Lambda(g)\setminus \Lambda(h^{-1}))$
and $\Lambda(h)\setminus(h^{-1}\cdot (-\Lambda(g)))$. But
$$|h^{-1}\cdot(\Lambda(g)\setminus \Lambda(h^{-1}))| = |\Lambda(g)| -
|\Lambda(g) \cap \Lambda(h^{-1})|$$ and
\begin{align*} |\Lambda(h)\setminus(h^{-1}\cdot (-\Lambda(g)))| &= |(-h\cdot
\Lambda(h)) \setminus \Lambda(g)| \\ &= |\Lambda(h^{-1})\setminus
\Lambda(g)| \\ &= |\Lambda(h^{-1})| - |\Lambda(h^{-1})\cap
\Lambda(g)| \\ &=
|\Lambda(h)| - |\Lambda(g) \cap \Lambda(h^{-1})|.\end{align*} %
Therefore $|\Lambda(gh)| = |\Lambda(g)| + |\Lambda(h)| -
2|\Lambda(g) \cap \Lambda(h^{-1})|$, as required.\end{proof}\\

Since this paper is concerned with involutions, let us say a few
words about conjugacy classes of involutions. It is well known that
involutions in $W(A_{n-1}) \cong \sym(n)$ are parameterised by cycle
type. That is, involutions are conjugate if and only if they have
the same number of transpositions $m$. Therefore Theorem \ref{sn} is
giving the length polynomials for individual conjugacy classes of
involutions. For type $B_n$, the situation is similar. It is clear
that an element $w$ of $W(B_n)$ is an involution precisely when each
of its signed cycles is either a positive transposition, a negative
1-cycle or a positive 1-cycle. It can be shown that conjugacy
classes are parameterised by signed cycle type, so that involutions
are conjugate if and only if they have the same number, $m$, of
transpositions (all of which must be positive), and the same number,
$e$, of negative 1-cycles. So again, Theorem \ref{bn} gives the
length polynomials for individual conjugacy classes of involutions.
The case for type $D_n$ is slightly more involved. Involutions again
consist of 1-cycles and positive transpositions, and since we are in
type $D_n$ there must be an even number of negative 1-cycles. There
is exactly one conjugacy class of involutions in $W(D_n)$ with $m$
transpositions and $e$ negative 1-cycles (with $e$ even), {\em
except} when $n = 2m$. In that case, there are two conjugacy
classes. However, there is a length preserving automorphism of the
Coxeter graph which interchanges these classes. Therefore if $X$ is
either class in $W(D_n)$, we see that $L_{2m, m, 0}(t) =
2L_{2m,X}(t)$, and so the information for any conjugacy class of
involutions can be retrieved from Theorem \ref{dn}. This means that
Theorems \ref{sn}, \ref{bn} and \ref{dn} give full information about
the length polynomials of conjugacy classes of involutions in the
classical Weyl groups. This is the motivation for finding, in
Section 4, corresponding information for conjugacy classes of
involutions in
the remaining finite irreducible Coxeter groups.\\

Here is a rough outline of the method we will follow in the next two
sections. Let $W_n$ be of type $A_{n-1}, B_n$ or $D_n$, where we are
viewing types $A_{n-1}$ and $D_n$ as subgroups of $W(B_n)$ as
described above. For an involution $x$ of $W$, we will define $\tau$
to be the cycle containing $n$ in the expression for $x$ as a
product of disjoint cycles. Since $x$ is an involution, $\tau$ will
be either a 1-cycle or a transposition. Then set $y = \tau x$. By
calculating the length of $y$ in terms of the lengths of $\tau$ and
$x$, we may hope to obtain a recursive formula for our length
polynomial. However if $\tau$ is a transposition, we cannot retrieve
$\ell(x)$ from $\ell(y)$ without knowing $y$ itself, and thus we
must perform another step. That step is to `compress' $y$ by
conjugating by $(\+ n\; \+{n-1}\; \cdots \+ r)$. This results in an
involution $z$ whose length depends only on $n$, $r$ and $\ell(x)$.

\section{Results for Classical Weyl groups}

Throughout this section, $W$ will be a Coxeter group of type
$A_{n-1}$, $B_n$ or $D_n$.

\begin{thm}\label{tech1}
Let $x$ be an involution in $W$, $\tau$ be the cycle containing $n$
in the expression for $x$ as a product of disjoint cycles, and set
$y = x\tau$. Then either $\tau = (\+ n)$, $\tau = (\mi n)$ or there
is some $r$ with $1 \leq r < n$ for which $\tau$ equals $(\+ r \+
n)$ or $(\mi r \mi n)$. Moreover, writing $$\Delta_r(y) = |\{k \in
\{1, \ldots, n\}: |y(k)| < r < k\}| + |\{k \in \{1, \ldots, n\}:
y(k) < 0, r < k, r < |y(k)|\}|,$$ we have %
\begin{align*} |\Sigma(x)| &= \left\{\begin{array}{ll} |\Sigma(y)| &
{\mbox{ if $\tau = (\+ n)$}}\\
|\Sigma(y)| + 1 &
{\mbox{ if $\tau = (\mi n)$}}\\
|\Sigma(y)|  &
{\mbox{ if $\tau = (\+ r\+n)$}}\\
|\Sigma(y)| + 2 & {\mbox{ if $\tau = (\mi r \mi
n)$}}\end{array}\right.\\
|\Lambda(x)| &= \left\{\begin{array}{ll} |\Lambda(y)| &
{\mbox{ if $\tau = (\+ n)$}}\\
2(n-1) + |\Lambda(y)| &
{\mbox{ if $\tau = (\mi n)$}}\\
2(n-r) - 1 + |\Lambda(y)| - 2\Delta_r(y) &
{\mbox{ if $\tau = (\+ r\+n)$}}\\
2(n+r) - 5 + |\Lambda(y)| - 2\Delta_r(y) & {\mbox{ if $\tau = (\mi r
\mi n)$}}\end{array}\right.\end{align*}\end{thm}

\begin{proof} Since for any $w \in W$,  $\Sigma(w)$ is just the number of minus signs
in the expression for $w$ as a product of disjoint signed cycles,
the result for $\Sigma(x)$ is clear. Therefore we will now
concentrate on finding $\Lambda(x)$ in terms of $\Lambda(y)$. By
Lemma \ref{ngh} we observe that

\begin{equation}\label{eq1}|\Lambda(x)| = |\Lambda(\tau)| + |\Lambda(y)| - 2|\Lambda(y)
\cap \Lambda(\tau)|.\end{equation}

Now $y(n) = n$ and, if $\tau$ is not $(\+ n)$ or $(\mi n)$, then
$y(r) = r$. Therefore for $1 \leq i < n$ we see that $y\cdot (e_n
\pm e_i)$ is either $e_n + e_j$ or $e_n - e_j$ for
some $j < n$. Thus $e_n \pm e_i \notin \Lambda(y)$. We now work through the possibilities for $\tau$.\\

If $\tau = (\+ n)$, then $x = y$ and there is nothing to prove.\\

Suppose that $\tau = (\mi n)$. Then $$\Lambda(\tau) = \{e_n \pm e_i:
1 \leq i < n\}.$$ Thus $|\Lambda(\tau)| = 2(n-1)$ and $\Lambda(\tau)
\cap \Lambda(y) = \emptyset$. Substituting these values into
Equation \ref{eq1} gives $$|\Lambda(x)| = 2(n-1) + |\Lambda(y)|.$$

We move on to the case when $\tau = (\+ r \+ n)$. Then
$$\Lambda(\tau) = \{e_j - e_r: r < j < n\} \cup \{e_n - e_i: i \leq
r < n\}.$$ Thus $|\Lambda(\tau)| = 2(n-r) - 1$. We have already
noted that $e_n - e_i \notin \Lambda(y)$. Now $y\cdot(e_j - e_r)$ is
$e_{|y(j)|} - e_r$ if $y(j)
> 0$ and $-e_{|y(j)|} - e_r$ if $y(j) < 0$. If $|y(j)| < r < j$,
then $y\cdot (e_j - e_r) = \pm e_{|y(j)|} - e_r < 0$ and so $e_j -
e_r \in \Lambda(y)$. If $r <|y(j)|$, then $e_j - e_r \in \Lambda(y)$
if and only if $r < j$ and $y(j) < 0$. Hence $$|\Lambda(y) \cap
\Lambda(\tau)| = |\{j : |y(j)|
< r < j\}| + |\{j: r<j, r<|y(j)|, y(j) < 0\}| = \Delta_r(y).$$ %
Substitution into Equation 1 now produces the expression
$$|\Lambda(x)| = 2(n-r) - 1 + |\Lambda(y)| - 2\Delta_r(y).$$
Finally we set $\tau = (\mi r \mi n)$. Then
\begin{align*}\Lambda(\tau) &= \{e_n \pm e_i: 1\leq i < r\} \cup \{e_n
- e_i: r < i < n\} \\ & \hspace*{1cm} \cup \{e_r \pm e_i: 1 \leq i <
r\} \cup \{e_j + e_r: r < j \leq n\}.\end{align*} Hence
$|\Lambda(\tau)| = 2(n+r) - 5$, and for $\Lambda(\tau) \cap
\Lambda(y)$ we need only consider roots of the form $e_r \pm e_i$,
for $i < r$, and roots $e_j + e_r$ for $r < j$. For $e_r \pm e_i$,
we have $y \cdot \{e_r + e_i, e_r - e_i\} = \{e_r + e_{|y(i)|}, e_r
- e_{|y(i)|}\}$. Now $e_r + e_{|y(i)|}$ is certainly positive, and
$e_r - e_{|y(i)|}$ is negative precisely when $r < |y(i)|$. Hence
$$|\Lambda(y) \cap \{e_r \pm e_i: 1 \leq i < r\}| = |\{i: i < r <
|y(i)|\}|.$$ Moreover, we see that $e_j + e_r \in \Lambda(y)$
precisely when $r < j$ (to ensure $e_j + e_r$ is positive), $y(j) <
0$ and $r < |y(j)|$. Therefore $$|\Lambda(y) \cap \Lambda(\tau)| =
|\{i: i < r < |y(i)|\}| + |\{j: r < j, y(j) < 0, r < |y(j)|\}| =
\Delta_r(y).$$ A final substitution into Equation 1 gives
$$|\Lambda(x) = 2(n+r) - 5 + |\Lambda(y)| - 2\Delta_r(y),$$ and this completes the
proof of Theorem \ref{tech1}.
\end{proof}\\

The next result uses $\Delta_r(y)$ again. The definition is the same
as that given in Theorem \ref{tech1}, but is included in the
statement of Theorem \ref{tech2} for ease of reference. We need one
more definition.

\begin{defn} For $r < n$, we define $c_r$ to be the cycle $(\+ n \;\; \+{n-1} \;
\cdots \; \+ r)$.\end{defn}

In Theorem \ref{tech2} note that by $y^{c_r}$ we mean
$c_ryc_r^{-1}$.

\begin{thm}\label{tech2} Let $y$ be an involution in $W$ with the property that
$y(r) = r$ for some $r < n$. Then $|\Sigma(y^{c_r})| = |\Sigma(y)|$.
Moreover, writing $$\Delta_r(y) = |\{k \in \{1, \ldots, n\}: |y(k)|
< r < k\}| + |\{k \in \{1, \ldots, n\}: y(k) < 0, r < k, r <
|y(k)|\}|,$$ we have
$$|\Lambda(y^{c_r})| = |\Lambda(y)| - 2\Delta_r(y).$$\end{thm}

\begin{proof}
Let $y$ be an involution in $W$ with $y(r) = r$. We will write $c$
instead of $c_r$ for ease of notation. Since $c$ contains no minus
signs, clearly $|\Sigma(y)| = |\Sigma(y^{c})|$. To derive the result
for $\Lambda(y)$, we will consider two subsets $V_r$ and $U_r$ of
$\Lambda(y)$, where
\begin{align*} V_r(y) &= \Lambda(y) \cap \{e_n \pm e_r, \ldots,
e_{r+1} \pm e_r, e_r \pm
e_{r-1}, \ldots, e_r \pm e_1\} \; {\mbox{ and}}\\
U_r(y) &= \Lambda(y) \setminus V_r.\end{align*}%
Note that $\Lambda(y)$ is the disjoint union of $U_r(y)$ and
$V_r(y)$. We claim that $\Lambda(y^{c}) = c\cdot U_r(y)$. Firstly,
consider $c\cdot U_r(y)$. A root $e_j \pm e_i$ is in $U_r(y)$ if and
only if $r \notin\{i,j\}$, $j > i$ and $y\cdot(e_j \pm e_i) < 0$.
Now $c$ is order preserving on $\{1, \ldots, r, r+1, \ldots, n\}$,
and $c^{-1}$ is order preserving on $\{1, \ldots, n-1\}$. Hence $e_j
\pm e_i \in U_r(y)$ if and only if $n \notin \{c(i), c(j)\}$, $c(j)
> c(i)$ and $y\cdot(e_j \pm e_i) < 0$. Now $y(r) = r$, so if $r \notin
\{i,j\}$, then $r \notin \{|y(i)|, |y(j)|\}$. Hence $e_j \pm e_i \in
U_r(y)$ if and only if $n \notin \{c(i), c(j)\}$, $c(j) > c(i)$ and
$y^{c}\cdot (e_{c(j)} \pm e_{c(i)}) = c y\cdot(e_j \pm e_i) < 0$.
That is, $e_j \pm e_i \in U_r(y)$ if and only if $n \notin \{c(i),
c(j)\}$ and $e_{c(j)} \pm e_{c(i)} \in \Lambda(y^{c})$. Observe
though that $y^{c}(n) = n$, and so $\Lambda(y^{c})$ contains no
elements of the form $e_n \pm e_i$. Therefore the restriction $n
\notin\{c(i), c(j)\}$ is redundant for elements of $\Lambda(y^{c})$.
Therefore $e_j \pm e_i \in U_r(y)$ if and only if $e_{c(j)} \pm
e_{c(i)} \in \Lambda(y^c)$. That is, $\Lambda(y^{c}) = c\cdot
U_r(y)$, as claimed.\\

We have shown so far that $|\Lambda(y)| = |\Lambda(y^c)| +
|V_r(y)|$. So it remains to find $|V_r(y)|$. Unfortunately there are
eight possibilities. For $r<j$, we must look at the positive roots
$e_j - e_r$ and $e_j + e_r$. For $i < r$, we must look at $e_r -
e_i$ and $e_r + e_r$. The following tables give the outcome in each
case. Firstly, we consider $j$ where $r<j$.

\begin{center} \begin{tabular}{ccrrc}
 $y(j)$ & $|y(j)|$ & $y\cdot (e_j - e_r)$ & $y\cdot (e_j + e_r)$ & \# Roots in $V_r(y)$\\
\hline
$<0$ & $<r$ & $-e_{|y(j)|} - e_r$ & $-e_{|y(j)|} + e_r$& 1\\
$<0$ & $>r$ & $-e_{|y(j)|} - e_r$ & $-e_{|y(j)|} + e_r$& 2\\
$>0$ & $<r$ & $e_{|y(j)|} - e_r$ & $e_{|y(j)|} + e_r$& 1\\
$>0$ & $>r$ & $e_{|y(j)|} - e_r$ & $e_{|y(j)|} + e_r$& 0
\end{tabular}\end{center}

Therefore the number of roots in $V_r(y)$ of the form $e_j \pm e_r$
for some $j > r$ is \begin{equation} \label{eq2} |\{j: |y(j)| < r <
j \}| + 2|\{j: r < j, y(j) < 0, r < |y(j)|\}|.\end{equation}

Now we consider $i < r$ in the following table.

\begin{center} \begin{tabular}{ccllc}
 $y(i)$ & $|y(i)|$ & $y\cdot (e_r - e_i)$ & $y\cdot (e_r + e_i)$ & \# Roots in $V_r(y)$\\
\hline
$<0$ & $<r$ & $e_r + e_{|y(i)|}$ & $e_r - e_{|y(i)|}$& 0\\
$<0$ & $>r$ & $e_r + e_{|y(i)|}$ & $e_r - e_{|y(i)|}$& 1\\
$>0$ & $<r$ & $e_r - e_{|y(i)|}$ & $e_r + e_{|y(i)|}$& 0\\
$>0$ & $>r$ & $e_r - e_{|y(i)|}$ & $e_r + e_{|y(i)|}$& 1
\end{tabular}\end{center}

Therefore the number of roots in $V_r(y)$ of the form $e_r \pm e_i$
for some $i < r$ is $|\{i: i < r < |y(i)|\}|$. Writing $k = |y(i)|$,
we note that $\{i: i < r < |y(i)|\} = \{k: |y(k)| < r < k\}$.
Therefore the number of roots in $V_r(y)$ of the form $e_r \pm e_i$
for some $i < r$ is \begin{equation}\label{eq3}|\{k: |y(k)| < r <
k\}|.\end{equation} Combining (\ref{eq2}) and (\ref{eq3}) we get
that $$|V_r(y)| = 2|\{k \in \{1, \ldots, n\}: |y(k)| < r < k\}| +
2|\{k \in \{1, \ldots, n\}: y(k) < 0, r < k, r < |y(k)|\}|,$$ which
is just $2\Delta_r(y)$. Recalling that $|\Lambda(y)| = |U_r(y)| +
|V_r(y)|$ gives $|\Lambda(y)| = |\Lambda(y^c)| + 2\Delta_r(y)$, and
the proof is complete.
\end{proof}

\begin{cor} \label{tech3} Let $x$ be an involution in $W$ such that $\tau = (\+r \+n)$ is
the cycle containing $n$ in the expression of $x$ as a product of
disjoint signed cycles. Let $z = (x\tau)^{c_r}= c_r(x\tau)c_r^{-1}$.
Then $|\Sigma(x)| = |\Sigma(z)|$ and $|\Lambda(x)| = |\Lambda(z)| +
2(n-r) - 1$.\end{cor}

\begin{proof} Let $y = x\tau$. Then by Theorem \ref{tech1},
$|\Sigma(y)| = |\Sigma(x)|$ and $\Lambda(x) = 2(n-r)-1 +
|\Lambda(y)| - 2\Delta_r(y)$. By Theorem \ref{tech2}, $|\Sigma(z)| =
|\Sigma(y)|$ and $|\Lambda(z)| = |\Lambda(y)| - 2\Delta_r(y)$. The
result follows immediately.\end{proof}\\

An almost identical argument to the proof of Corollary \ref{tech3}
gives Corollary \ref{tech4}.

\begin{cor}\label{tech4} Let $x$ be an involution in $W$ such that $\tau = (\mi r \mi n)$ is
the cycle containing $n$ in the expression of $x$ as a product of
disjoint signed cycles. Let $z = (x\tau)^{c_r}= c_r(x\tau)c_r^{-1}$.
Then $|\Sigma(x)| = |\Sigma(z)| + 2$ and $|\Lambda(x)| =
|\Lambda(z)| + 2(n+r) - 5$.\end{cor}

\section{Proof of the Main Theorems}

\paragraph{Proof of Theorem \ref{sn}} Let $X_{A_{n-1},m,\ell}$ be the set
of involutions in $W(A_{n-1})\cong \sym(n)$ with $m$ transpositions
and length $\ell$. Let $X_{A_{n-1},m,\ell,r}$ be the set of
involutions $x$ in $W(A_{n-1})$ with $m$ transpositions and length
$\ell$  such that $x(n) = r$. Then $X_{A_{n-1}, m, \ell} =
\cup_{r=1}^n X_{A_{n-1}, m, \ell, r}$. Clearly $X_{A_{n-1}, m, \ell,
n} = X_{A_{n-2}, m, \ell}$. If $r < n$, then the cycle containing
$n$ is $(\+ r \+ n)$. Therefore, by Theorem \ref{ls} and Corollary
\ref{tech3}, the map $x \mapsto (x(\+ r\+ n))^{c_r}$ is a bijection
between $X_{A_{n-1},m,\ell,r}$ and $X_{A_{n-3},m-1,\ell-2(n-r) +
1}$. Hence
$$|X_{A_{n-1}, m, \ell}| = |X_{A_{n-2},m,\ell}| + \sum_{r=1}^{n-1}
|X_{A_{n-3},m-1,\ell-2(n-r)+1}|.$$ From this, setting $k=n-r$, we
immediately get $$\alpha_{n, m, \ell} = \alpha_{n-1, m, \ell} +
\sum_{k=1}^{n-1} \alpha_{n-2,m-1,\ell + 1 - 2k},$$ which is the
first part of Theorem \ref{sn}, and
\begin{align*} L_{n,m}(t) &= L_{n-1, m}(t) + \sum_{k=1}^{n-1}
t^{2k-1}L_{n-2,m-1}(t) \\
&= L_{n-1,m}(t) + (t + t^3 + \cdots + t^{2n-3})L_{n-2,m-1}(t)\\
&= L_{n-1,m}(t) +
\frac{t(t^{2(n-1)}-1)}{t^2-1}L_{n-2,m-1}(t).\end{align*} This gives
the second statement in Theorem \ref{sn}.\qed

We note that the statement relating to type $A_{n-1}$ in Corollary
\ref{all} is a simple consequence of Theorem \ref{sn} and the fact
that $L_{W(A_{n-1})}(t) = \sum_{m=1}^{\lfloor n/2 \rfloor}
L_{n,m}(t)$.

\paragraph{Proof of Theorem \ref{bn}} Let $X_{B_{n},m,e,\ell}$ be the set
of involutions of length $\ell$ in $B_n$ whose expression as a
product of disjoint signed cycles has $m$ transpositions and $e$
negative 1-cycles. Let $X_{B_{n},m,e,\ell,\rho}$ be the set of
involutions $x$ in $X_{B_n,m,e,\ell}$ such that $x(n) = \rho$. Then
$X_{B_{n}, m, \ell} = \bigcup_{r=1}^n \left(X_{B_{n}, m,e, \ell, r}
\cup X_{B_n,m,e,\ell,(-r)}\right)$. Clearly $X_{B_n, m, e, \ell, n}
= X_{B_{n-1}, m, e, \ell}$. If $\rho = -n$, then by Theorem \ref{ls}
and Theorem \ref{tech1}, the map $x \mapsto x(\mi n)$ is a bijection
between $X_{B_n,m,e,\ell,(-n)}$ and $X_{B_{n-1},m,e-1,\ell+1-2n}$.
If $\rho = r < n$, then by Theorem \ref{ls} and Corollary
\ref{tech3}, $x \mapsto (x(\+ r \+n))^{c_r}$ is a bijection between
$X_{B_n,m,e,\ell,r}$ and $X_{B_{n-2},m-1,e,\ell+1-2(n-r)}$. Finally
if $\rho = -r \neq -n$, then by Theorem \ref{ls} and Corollary
\ref{tech4}, $x \mapsto (x(\mi r \mi n))^{c_r}$ is a bijection
between $X_{B_n,m,e,\ell,r}$ and $X_{B_{n-2},m-1,e,\ell+3-2(n+r)}$.

Hence \begin{align*} |X_{B_n,m,e,\ell}| &= |X_{B_{n-1},m,e,\ell}| +
|X_{B_{n-1},m,e-1,\ell+1-2n}| \\ & \;\;\; +
\sum_{r=1}^{n-1}\left(|X_{B_{n-2},m-1,e,\ell+1-2(n-r)}| +
|X_{B_{n-2},m-1,e,\ell+3-2(n+r)}|\right)\\
&= |X_{B_{n-1},m,e,\ell}| + |X_{B_{n-1},m,e-1,\ell+1-2n}| +
\sum_{k=1}^{2n-2} |X_{B_{n-2},m-1,e,\ell + 1 - 2k}|.\end{align*}

From the definitions of $\beta_{n,m,e,\ell}$  and $L_{n,m,e}(t)$ we
now
get%
$$\beta_{n, m,e,\ell} = \beta_{n-1,m,e,\ell} + \beta_{n-1,m,e-1,\ell+1-2n} +
\sum_{k=1}^{2n-2} \beta_{n-2,m-1,e,\ell + 1 - 2k},$$ which is the
first part of Theorem \ref{bn}, and
\begin{align*} L_{n,m,e}(t) &= L_{n-1, m,e}(t) + t^{2n-1}L_{n-1,m,e-1}(t) +
 \sum_{k=1}^{2n-2}
t^{2k-1}L_{n-2,m-1,e}(t)\\
&= L_{n-1, m,e}(t) + t^{2n-1}L_{n-1,m,e-1}(t) + (t + t^3 + \cdots +
t^{4n-5})
L_{n-2,m-1,e}(t) \\
&= L_{n-1, m,e}(t) + t^{2n-1}L_{n-1,m,e-1}(t) +
\frac{t(t^{4n-4}-1)}{t^2-1} L_{n-2,m-1,e}(t).\end{align*} This gives
the second statement in Theorem \ref{bn}.\qed

Before embarking on the case of $D_n$, we define yet another
polynomial. Let $Y_{B_n,m,e}$ be the set of involutions in $W(B_n)$
whose expression as a product of disjoint signed cycles has $m$
transpositions and $e$ negative 1-cycles. Set

$$E_{n,m,\ell}(t) = \sum_{x \in Y_{B_n,m,e}} t^{|\Lambda(x)|}.$$

\paragraph{Proof of Theorem \ref{dn}} We need to work within the
environment of $B_n$ for the moment, because of the risk that when
we remove the cycle containing $n$ from an involution that happens
to be in the subgroup $W(D_n)$, we end up with an involution outside
of $W(D_n)$. We get round this by working in $W(B_n)$, but instead
of considering $\ell_{B_n}(x)$ or $\ell_{D_n}(x)$ for an involution
$x$, we consider $|\Lambda(x)|$. If $x$ happens to be an element of
$W(D_n)$, then by Theorem \ref{ls}, $\ell_{D_n}(x) = |\Lambda(x)|$.
Therefore we will be able, with care, to retrieve the length
polynomial for
involutions in $W(D_n)$ at the end of the process.\\

Let $Y_{B_{n},m,e,\ell}$ be the set of involutions $x$ of $W(B_n)$
satisfying $|\Lambda(y)| = \ell$, whose expression as a product of
disjoint signed cycles has $m$ transpositions and $e$ negative
1-cycles. Let $Y_{B_{n},m,e,\ell,\rho}$ be the set of involutions
$x$ in $Y_{B_n,m,e,\ell}$ such that $x(n) = \rho$. Then $$Y_{B_{n},
m, e, \ell} = \bigcup_{r=1}^n \left(Y_{B_{n}, m,e, \ell, r} \cup
Y_{B_n,m,e,\ell,(-r)}\right).$$ Clearly $|Y_{B_n, m, e, \ell, n}| =
|Y_{B_{n-1}, m, e, \ell}|$. If $\rho = -n$, then by Theorem \ref{ls}
and Theorem \ref{tech1}, the map $x \mapsto x(\mi n)$ is a bijection
between $Y_{B_n,m,e,\ell,(-n)}$ and $Y_{B_{n-1},m,e-1,\ell+2-2n}$.
If $\rho = r < n$, then by Theorem \ref{ls} and Corollary
\ref{tech3}, $x \mapsto (x(\+ r \+n))^{c_r}$ is a bijection between
$Y_{B_n,m,e,\ell,r}$ and $Y_{B_{n-2},m-1,e,\ell+1-2(n-r)}$. Finally
if $\rho = -r \neq -n$, then by Theorem \ref{ls} and Corollary
\ref{tech4}, $x \mapsto (x(\mi r \mi n))^{c_r}$ is a bijection
between $Y_{B_n,m,e,\ell,r}$ and $Y_{B_{n-2},m-1,e,\ell+5-2(n+r)}$.

Hence \begin{align*} |Y_{B_n,m,e,\ell}| &= |Y_{B_{n-1},m,e,\ell}| +
|Y_{B_{n-1},m,e-1,\ell+2-2n}| \\ & \;\;\; +
\sum_{r=1}^{n-1}\left(|Y_{B_{n-2},m-1,e,\ell+1-2(n-r)}| +
|Y_{B_{n-2},m-1,e,\ell+5-2(n+r)}|\right)\\
&= |Y_{B_{n-1},m,e,\ell}| + |Y_{B_{n-1},m,e-1,\ell+2-2n}| \\ & \;\;+
\sum_{k=1}^{n-1} \left(|Y_{B_{n-2},m-1,e,\ell + 1 - 2k}| +
|Y_{B_{n-2}, m-1, e, \ell+5 - 2n-2k}|\right).\end{align*}

We now consider the polynomial $E_{n,m,e}(t)$
 with the aim of showing that $E_{n,m,e}(t)$ is
precisely the $D_{n,m,e}(t)$ defined just before Theorem \ref{dn}.
We observe that $$E_{n,m,e}(t) = \sum_{\ell=0}^{\infty}
|Y_{B_n,m,e,\ell}|t^\ell.$$ Certainly $E_{n,0,0}(t) = 1$ and if $2m
+ e > n$, then $E_{n,m,e}(t) = 0$. Now $W(B_1) = \{(\+ 1), (\mi
1)\}$, so $E_{1,0,1}(t) = t$. The set of involutions in $W(B_2)$ is
$\{(\+ 1 \+ 2), (\mi 1 \mi 2), (\mi 1), (\mi 2), (\mi 1)(\mi 2)\}$.
Hence $E_{2,1,0}(t) = 2t$, $E_{2,0,1}(t) =  1 + t^2$, $E_{2,0,2}(t)
= t^2$. Moreover, from our recurrence relation for $|Y_{B_n, m, e,
\ell}|$ above, we get that for
$n \geq 3$, %

\begin{align*} E_{n,m,e}(t) &= E_{n-1, m,e}(t) +
t^{2n-2}E_{n-1,m,e-1}(t) +
 \sum_{k=1}^{n-1}
\left(t^{2k-1} + t^{2k + 2n-5}\right)E_{n-2,m-1,e}(t)\\
&= E_{n-1, m,e}(t) + t^{2n-2}E_{n-1,m,e-1}(t) +
(1+t^{2n-4})\sum_{k=1}^{n-1}t^{2k-1}E_{n-2,m-1,e}(t) \\
&= E_{n-1, m,e}(t) + t^{2n-2}E_{n-1,m,e-1}(t) +
\frac{t(1+t^{2n-4})(t^{2n-2}-1)}{t^2-1}
E_{n-2,m-1,e}(t).\end{align*} That is,

\begin{equation} \label{eq4} E_{n,m,e}(t) = E_{n-1, m,e}(t) + t^{2n-2}E_{n-1,m,e-1}(t) +
\frac{t(1+t^{2n-4})(t^{2n-2}-1)}{t^2-1}
E_{n-2,m-1,e}(t).\end{equation}

Now $E_{n,m,e}(t)$ has exactly the same initial conditions and
recurrence relation as $D_{n,m,e}(t)$. So the two polynomials are
the same. In particular, the coefficient of $t^\ell$ in
$D_{n,m,e}(t)$ equals the coefficient of $t^\ell$ in $E_{n,m,e}(t)$,
which by definition is $|Y_{B_n,m,e,\ell}|$. But we know that when
$e$ is even, the elements of $Y_{B_n,m,e,\ell}$ are precisely the
elements of $W(D_n)$ with length $\ell$ that have $m$ transpositions
and $e$ negative 1-cycles. Therefore $|Y_{B_n,m,e,\ell}| =
\delta_{n,m,e,\ell}$ and hence the coefficient of $t^{\ell}$ in
$D_{n,m,e}$ is $\delta_{n,m,e,\ell}$ when $e$ is even. If $e$ is
odd, then there are no involutions in $W(D_n)$ with $e$ negative
1-cycles, so $\delta_{n,m,e,\ell} = 0$. Similarly, when $e$ is odd
we have $L_{n,m,e}(t) = 0$. When $e$ is even, the fact that
$|Y_{B_n,m,e,\ell}| = \delta_{n,m,e,\ell}$ implies that
$D_{n,m,e}(t) = E_{n,m,e}(t) = L_{n,m,e}(t)$. This completes the
proof of Theorem \ref{dn}.\qed

Given that we now have expressions for the length polynomials for
involutions of every signed cycle type in $W(A_n)$, $W(B_n)$ and
$W(D_n)$, we can now produce recurrence relations for the length
polynomials $L_{W}(t)$ for the sets of all involutions in these
groups. The only potential stumbling block is $W(D_n)$. Here our
recurrence relation for $L_{n,m,e}(t)$ involves involutions with
$e-1$ negative 1-cycles, which of course are not elements of
$W(D_n)$. We work round this as follows. Define $$L_{(B\setminus
D)_n}(t) = \sum_{m, j \geq 0} E_{n,m,2j+1}(t).$$ Observe that the
coefficient of $t^{\ell}$ in $L_{(B \setminus D_n)}(t)$ is the
number of involutions $x$ in $W(B_n)$ for which $|\Lambda(x)| =
\ell$ whose expression as a product of disjoint signed cycles
contains an odd number of negative 1-cycles. These are precisely the
involutions of $W(B_n)$ which are not contained in $W(D_n)$. We may
now state and prove Corollary \ref{all}. Note that part (a) of
Corollary \ref{all} is known; it is the first
part of Proposition 2.8 in \cite{dukes}.\\

\begin{cor}\label{all} \begin{enumerate} \item[(a)]  $L_{W(A_1)}(t) = 1 + t$,
$L_{W(A_2)}(t) = 1 + 2t + t^3$, and for $n \geq 3$,
$$L_{W(A_{n})}(t) = L_{W(A_{n-1})}(t) +
\frac{t(t^{2n}-1)}{t^2-1}L_{W(A_{n-2})}(t).$$%
\item[(b)] $L_{W(B_1)}(t) = 1 + t$, $L_{W(B_2)}(t) = 1+ 2t + 2t^3 + t^4$
and for $n \geq 3$,
$$L_{W(B_n)}(t) = (1 + t^{2n-1})L_{W(B_{n-1})}(t) + \frac{t(t^{4n-4)}-1)}{t^2-1} L_{W(B_{n-2})}(t).$$
\item[(c)] $L_{W(D_1)}(t) = 1$, $L_{(B\setminus D)_1}(t) = 1$, $L_{W(D_2)}(t) = 1 + 2t$,
$L_{(B\setminus D)_2}(t) = 1 + t^2$ and for $n \geq 3$, $$L_{D_n}(t)
=  L_{W(D_{n-1})}(t) + t^{2n-2}L_{(B\setminus D)_{n-1}}(t) +
\frac{t(1+t^{2n-4})(t^{2n-2}-1)}{t^2-1} L_{W(D_{n-2})}(t)
$$ and $$L_{(B\setminus D)_n}(t) = L_{(B\setminus D)_{n-1}}(t) + t^{2n-2}L_{W(D_{n-1})}(t) +
\frac{t(1+t^{2n-4})(t^{2n-2}-1)}{t^2-1} L_{(B\setminus
D)_{n-2}}(t).$$

\end{enumerate}
\end{cor}

\begin{proof} For parts (a) and (b), we simply observe that
$L_{W(A_n)}(t) = \sum_m L_{n+1,m}(t)$, and similarly $L_{W(B_n)}(t)
= \sum_{m,e} L_{n,m,e}(t)$, and apply Theorems \ref{sn} and
\ref{bn}. For part (c), the initial values $L_{W(D_1)}(t)$,
$L_{(B\setminus D)_1}(t)$, $L_{W(D_2)}(t)$ and $L_{(B\setminus
D)_2}(t)$ are easy to calculate. For the recurrence relations we use
Equation \ref{eq4}:
$$ E_{n,m,e}(t) = E_{n-1, m,e}(t) +
t^{2n-2} E_{n-1,m,e-1}(t) + \frac{t(1+t^{2n-4})(t^{2n-2}-1)}{t^2-1}
 E_{n-2,m-1,e}(t).$$ Hence
\begin{align*} \sum_{m,j} E_{n,m,2j}(t) &= \sum_{m,j} E_{n-1, m,2j}(t) +
t^{2n-2} \sum_{m,j} E_{n-1,m,2j-1}(t)\\
& \;\; + \frac{t(1+t^{2n-4})(t^{2n-2}-1)}{t^2-1} \sum_{m,j}
E_{n-2,m-1,2j}(t)\end{align*} and \begin{align*}\sum_{m,j}
E_{n,m,2j+1}(t) &= \sum_{m,j} E_{n-1, m,2j+1}(t) + t^{2n-2}
\sum_{m,j} E_{n-1,m,2j}(t) \\ & \;\; +
\frac{t(1+t^{2n-4})(t^{2n-2}-1)}{t^2-1} \sum_{m,j}
E_{n-2,m-1,2j+1}(t).\end{align*} Now the involutions in $W(D_n)$ are
precisely the involutions in $W(B_n)$ whose signed cycle expression
has an even number $e = 2j$ of negative 1-cycles. Therefore
$L_{W(D_n)}(t) = \sum_{m,j} E_{n,m,2j}(t)$ and $L_{(B\setminus
D)_n}(t) = \sum_{m,j} E_{n,m,2j+1}(t)$. Hence we can immediately
conclude that
$$L_{W(D_n)}(t)
=  L_{W(D_{n-1})}(t) + t^{2n-2}L_{(B\setminus D)_{n-1}}(t) +
\frac{t(1+t^{2n-4})(t^{2n-2}-1)}{t^2-1} L_{W(D_{n-2})}(t)
$$ and $$L_{(B\setminus D)_n}(t) = L_{(B\setminus D)_{n-1}}(t) + t^{2n-2}L_{W(D_{n-1})}(t) +
\frac{t(1+t^{2n-4})(t^{2n-2}-1)}{t^2-1} L_{(B\setminus
D)_{n-2}}(t).$$\end{proof}

\section{The remaining finite Coxeter groups}

We first deal with the groups of type $I_n$. Here $W(I_n)$ is the
dihedral group of order $2n$. The following lemma is easy to prove.
\begin{lemma}
Let $W = W(I_n)$. If $n$ is odd, then there is one conjugacy class
of  involutions, and its length polynomial is $t^n +
\frac{2t(1-t^n)}{1-t^2}$. If $n$ is even, there are three conjugacy
classes of  involutions. One consists of the unique central
involution and has length polynomial $t^{n}$. The other two
conjugacy classes both have length polynomial
$\frac{t(1-t^n)}{1-t^2}$. Therefore
$$L_{W}(t) =  1 + t^n + \frac{2t(1-t^n)}{1-t^2}. $$

\end{lemma}

The remaining exceptional finite Coxeter groups are types $E_6, E_7,
E_8$, $F_4$, $H_3$ and $H_4$. We have used the computer algebra
package {\sc Magma}\cite{magma} to calculate the length polynomials
here. For a conjugacy class $X$ in $W$, where $W$ is one of these
groups, we write $a_{W,X,\ell}$ for the coefficent of $t^{\ell}$ in
the length polynomial. That is,
$$L_{W,X}(t) = \sum_{x \in X} t^{\ell(x)} = \sum a_{W,X,\ell} t^{\ell}.$$
The sequence $[a_{W,X,\ell}]$ (starting at the smallest nonzero
term) we refer to as the `length profile' of $X$ in $W$. Involutions
in a given conjugacy class have lengths of the same parity (either
all odd length or all even length), and so when writing down the
length profiles we would get alternating zeros. We suppress these,
and write the `odd length profile' (the sequence $[a_{W,X,2k+1}]$)
or `even length profile' (the sequence $[a_{W,X,2k}]$) as
appropriate. As a small example, in the dihedral group of order 8,
each conjugacy class of reflections has length profile $[1,0,1]$,
where the smallest length is 1. So its odd length profile is
$[1,1]$. The involution length profile (including the identity) of
the whole dihedral group of order 8 is $[1,2,0,2,1]$, so its odd
length profile is $[2,2]$ and its even
length profile is $[1,0,1]$. \\

The length profiles of conjugacy classes in the exceptional groups
of types $E_6$, $E_7$, $E_8$, $F_4$, $H_3$ and $H_4$ are given in
Tables \ref{tablee6} -- \ref{tableh4}. For these groups $W$, it is
well known that every nontrivial involution in $W$ is conjugate to
the central involution of some standard parabolic subgroup. For each
conjugacy class $X$ of involutions, the class is indicated by giving
(up to isomorphism) the relevant standard parabolic subgroup, the
size of the class, and the minimum length $\ell_{\min}$ of elements
in the class. This information is nearly always enough to specify
$X$ uniquely. Where it is not (and this occurs in type $F_4$), the
length profiles of the given conjugacy classes are happily
identical, by virtue of the length-preserving automorphism of the
Coxeter graph.\\

We begin with the table for $W(E_6)$.

\begin{table}[hbt]
\begin{center}
\begin{tabular}{cccl}
\hline class & size & $\ell_{\min}$ & odd/even length
profile\\ \hline %
$A_1$ & 36 &  1 & [6,5,5,5,4,3,3,2,1,1,1]\\
$A_1^2$ & 270 & 2 & [10,15,21,28,31,30,31,28,22,18,16,10,6,3,1]\\
$A_1^3$ & 540 & 3 & [5, 10, 17, 28, 40, 48, 56, 60, 58, 53, 49, 41, 32, 22, 13, 6, 2]\\
$D_4$ & 45  & 12 & [1,2,3,4,5,5,5,5,5,4,3,2,1]\\ \hline
\end{tabular} \caption{Involutions in $W(E_6)$}\label{tablee6}
\end{center}
\end{table}

Note that types $E_7$, $E_8$, $F_4$, $H_3$ and $H_4$ all have
non-trivial centres. This means that multiplication by the central
involution will map a given conjugacy class $X$ to one of equal size
with the length profile reversed. We exploit this in Tables
\ref{tablee7} --
 \ref{tableh4}.

\afterpage{\clearpage}

\begin{table}[H!]
\begin{center}
\begin{tabular}{cccl}
\hline class & size & $\ell_{\min}$ & odd/even length
profile\\ \hline %
$A_1$ & 63 &  1 & [7,6,6,6,6,5,5,4,4,3,3,2,2,1,1,1,1]\\
$A_1^2$ & 945 & 2 & [15, 24, 34, 44, 55, 60, 67, 68, 71, 68, 68, 62,
59, 50, 44, 38,\\ &&& \;\; 35, 26, 20, 14, 10, 6, 4, 2,
1]\\
$A_1^3$ & 315 & 3 &
[1,2,4,6,9,11,14,16,19,20,22,22,23,22,22,20,19,16,14,11,9,6,4,2,1]\\
$A_1^3$ & 3780 & 3 &
[10,22,39,61,91,119,152,180,209,228,248,257,265,259,251,235,\\ &&&
\;\;  222,198,175,147,122,94,72,51,35,20,11,5,2]
\\
$A_1^4$ & 3780 & 4 & reverse of class $A_1^3$, size 3780\\
$D_4$ & 315 & 12 & reverse of class $A_1^3$, size 315\\
$D_4 \times A_1$ & 945 & 13 & reverse of class $A_1^2$\\
$D_6$ & 63 & 30 & reverse of class $A_1$\\
$E_7$ & 1 & 63 & [1] \\ \hline
\end{tabular}
\caption{Involutions in $W(E_7)$} \label{tablee7}
\end{center}
\end{table}

\begin{table}[h!]
\begin{center}
\begin{tabular}{cccl}
\hline class & size & $\ell_{\min}$ & odd/even length
profile\\ \hline %
$A_1$ & 120 &  1 & [ 8, 7, 7, 7, 7, 7, 7, 6, 6, 6, 6, 5, 5, 4, 4, 4,
4, 3, 3, 2, 2, 2, 2, \\ &&& \; 1, 1, 1, 1, 1, 1 ]
\\
$A_1^2$ & 3780 & 2 & [21,35,50,65,80,95,111,120,130,140,151,155,161,
160,161,162, \\
&&& \; 164,159,157,148,141,134,129,117,108,99,92,85,80,68,59,50,\\
&&& \; 42,34,28,22,18,14,11,8,6,4,3,2,1]\\
$A_1^3$ & 37800 & 3 & [21, 49, 89, 141, 205, 279, 369, 460, 556,
656, 766, 868, 973, \\ &&& \;  1065, 1154, 1237, 1320, 1383,
1443, 1482, 1510, 1521, \\
&&& \;  1528, 1510, 1483, 1442, 1399, 1346, 1295, 1225, 1153,
1072, \\
&&& \;  989, 896, 805, 711, 625, 540, 464, 391, 326, 265, 215,  \\
&&& \; 170,
131, 95, 67, 45, 30, 18, 10, 5, 2]\\
$A_1^4$ & 113400 & 4 & [7, 20, 43, 80, 135, 207, 303, 420, 559, 719,
907, 1112, 1337, \\
&&& \;  1571, 1819, 2078, 2352, 2621, 2892, 3152, 3404, 3634,
3849,\\
&&& \;  4027, 4175, 4283, 4365, 4412, 4434, 4412, 4365, 4283,
4175,\\
&&& \;  4027, 3849,  3634, 3404, 3152, 2892, 2621, 2352, 2078,
1819,\\
&&& \  1571, 1337, 1112, 907, 719, 559, 420, 303, 207, 135, 80, 43,
20, 7]
\\
$D_4$ & 3150 & 12 & [1, 2, 4, 7, 11, 15, 21, 27, 34, 41, 49, 56, 65,
73, 82, 90, 99,105,  \\ &&& \;\;  112, 117, 122, 124, 127, 127, 128,
127, 127, 124, 122, 117, 112,\\
&&& \; 105, 99, 90, 82, 73, 65, 56, 49, 41, 34, 27, 21, 15, 11, 7,
4, 2, 1]
\\
$D_4 \times A_1$ & 37800 & 13 & reverse of class $A_1^3$\\
$D_6$ & 3780 & 30 & reverse of class $A_1^2$\\
$E_7$ & 120 & 63 & reverse of class $A_1$\\
$E_8$ & 1  & 120 & [1] \\ \hline
\end{tabular}
 \caption{Involutions in $W(E_8)$}\label{tablee8}
\end{center}
\end{table}

\afterpage{\clearpage} \pagebreak

The table for $W(E_8)$ is Table \ref{tablee8}. Note that the longest
element maps the class corresponding to $D_4$ to itself, and the
class corresponding to $A_1^4$ to itself. Therefore the length
profile for these classes are symmetric. The same occurs in the
class corresponding to $A_1^2$ in $W(F_4)$, as shown in Table
\ref{tablef4}.

\begin{table}[hbt]
\begin{center}
\begin{tabular}{cccl}
\hline class & size & $\ell_{\min}$ & odd/even length
profile\\ \hline %
$A_1$ & 12 &  1 & [2,2,2,2,1,1,1,1]\\
&&& (two such classes)\\
$A_1^2$ & 72 &  2 & [3,4,6,10,9,8,9,10,6,4,3]\\
$B_2$ & 18 & 4 & [1,2,3,2,2,2,3,2,1]\\
$B_3$ & 12 & 9 & reverse of class $A_1$\\
&&& (two such classes)\\
$F_4$ & 1  & 24 & [1]\\ \hline
\end{tabular}
 \caption{Involutions in $W(F_4)$} \label{tablef4}
\end{center}
\end{table}

Finally, we deal with $W(H_3)$ and $W(H_4)$. In $W(H_4)$ the class
corresponding to $A_1^2$ is mapped to itself by the longest element,
so the length profile for this class is symmetric.

\begin{table}[hbt]
\begin{center}
\begin{tabular}{cccl}
\hline class & size & $\ell_{\min}$ & odd/even length
profile\\ \hline %
$A_1$ & 15 &  1 & [3, 3, 3, 2, 2, 1, 1]\\
$A_1^2$ & 15 &  2 & reverse of class $A_1$\\
$H_3$ & 1  & 15 & [1]\\ \hline
\end{tabular}
 \caption{Involutions in $W(H_3)$}\label{tableh3}
\end{center}
\end{table}

\begin{table}[hbt]
\begin{center}
\begin{tabular}{cccl}
\hline class & size & $\ell_{\min}$ & odd/even length
profile\\ \hline %
$A_1$ & 60 &  1 & [4,4,4,4,4,4,4,3,3,3,3,3,3,2,2,2,2,1,1,1,1,1,1]\\
$A_1^2$ & 450 &  2 & [3,5,7,9,11,14,18,18,18,19,21,23,25,23,22,\\
& & & 23,25,23,21,19,18,18,18,14,11,9,7,5,3]\\
$H_3$ & 60  & 15 & reverse of class $A_1$\\
$H_4$ & 1  & 15 & [1]\\ \hline
\end{tabular}
 \caption{Involutions in $W(H_4)$}\label{tableh4}
\end{center}
\end{table}

There are various conjectures and some results about the odd and
even involution length profiles for classical Weyl groups. We will
discuss these in Section 6. For that reason, we include here the
tables of odd and even length profiles for the exceptional groups.
The odd involution length profiles of the exceptional groups are
given in Table \ref{oddExc}. The even involution length profiles of
the exceptional groups are given in Table \ref{evenExc}.

\begin{table}
\begin{tabular}{ll}Type & Profile\\
\hline
$E_6$ & [6,10,15,22,32,43,51,58,61,59,54,49,41,32,22,13,6,2]\\
$E_7$ &
[7,17,30,49,73,105,136,172,204,237,261,286,301,315,317,312,\\
& 300,291,273,251,226,199,171,144,120,96,75,55,39,26,15,1]\\
$E_8$ &
[8,28,56,96,148,212,288,380,476,580,692,816,940,1072,1200,1328,1456,\\
& 1588,1712,1836,1948,2052,2148,2240,2316,2380, 2432,2472,2500,
2520,\\
& 2520,2500,2472,2432,2380,2316,2240,2148,2052,1948,1836,1712,1588,\\
& 1456,1328,1200,1072,940,816,692,580,476,380,288,212,148,
96,56,28,8]
\\
$F_4$ & [4,4,4,4,4,4,4,4,4,4,4,4]\\
$H_3$ & [3,3,3,2,2,1,1,1]\\
$H_4$ &
[4,4,4,4,4,4,4,4,4,4,4,4,4,4,4,4,4,4,4,4,4,4,4,4,4,4,4,4,4,4]
\\
$I_{2k}$ & [2,2, \ldots, 2,2]\\
$I_{2k+1}$ & [2,2,\ldots,2,1]\\
\end{tabular}\caption{Odd involution length profiles in exceptional groups}\label{oddExc}\end{table}

\begin{table}
\begin{tabular}{ll} Type & Profile\\
\hline
$E_6$ & [1,10,15,21,28,31,31,33,31,26,23,21,15,11,8,5,3,2,1]\\
$E_7$ &
[1,15,26,39,55,75,96,120,144,171,199,226,251,273,291,300,312,\\
& 317,315,301,286,261,237,204,172,136,105,73,49,30,17,7]\\
$E_8$
&[1,21,35,50,65,80,96,113,124,137,151,166,176,188,194,203,213,\\
& 223,228,236,238,242,247,252,251,253,255,258,262, 266,264, 266,\\
& 262,258,255,253,251,252,247,242,238,236,228,223,213,203,194,\\
& 188,176,166,151,137,124,113,96,80,65,50,35,21,1]\\
$F_4$ & [1,3,5,8,13,11,10,11,13,8,5,3,1]\\
$H_3$ & [1,1,1,2,2,3,3,3]\\
$H_4$ & [1, 3, 5, 7, 9, 11, 14, 18, 18, 18, 19, 21, 23, 25, 23, 22,\\
& 23, 25, 23, 21, 19, 18, 18, 18, 14, 11, 9, 7, 5, 3, 1 ]
\\
$I_{2k}$ & [1,0,\ldots,0,1]\\
$I_{2k+1}$ & [1]
\end{tabular}\caption{Even involution length profiles (including the identity element) in exceptional groups}\label{evenExc}\end{table}

\pagebreak

\section{Unimodality}

Various conjectures have been made concerning the unimodality and/or
log-concavity of the coefficients of involution length polynomials
in the classical groups. Note that a sequence $(x_i)_{i=1}^N$ of
non-negative integers is {\em unimodal} if for some $j$, $x_1  \leq
\cdots \leq x_{j-1} \leq x_j \geq x_{j+1} \geq \cdots \geq x_N$, and
{\em log-concave} if for all $i$ between $2$ and $N-1$ we have
$x_i^2 \geq x_{i-1}x_{i+1}$ (see \cite{stanley}). A log-concave
sequence is always unimodal, but not vice versa.

\begin{conj}[\cite{dukes} Conjecture 4.1 (ii) -- (iv)] \label{duck}
\begin{enumerate} \item[(a)] The even involution length profile and the
odd involution length profile of $W(A_n)$ are log-concave.
\item[(b)] The even involution length profile and
the odd involution length profile in $W(B_n)$ are unimodal.
\item[(c)] The even involution length profile and the odd involution
length profile in $W(D_n)$ are unimodal.
\end{enumerate}
\end{conj}

To test these conjectures, and to see if they can be generalised, we
have checked the classical groups of types $A_n$,
$B_n$ and $D_n$ up to $n=10$, and all the exceptional groups.\\

These are the only examples we know of for even or odd involution
length profiles in finite irreducible Coxeter groups that are not
unimodal:
\begin{enumerate} \item Even length involutions in $W(B_6)$. Using
Corollary \ref{all} the even length profile is calculated to be
$$[1,10,20,27,35,41,49,51,55,54,55,51,49,41,35,27,20,10,1].$$
\item Even length involutions in types $E_8$, $F_4$, $H_4$ and $I_{n}$, for $n$ even.
\end{enumerate}

These are the only examples we know of for conjugacy classes of
involutions in finite irreducible Coxeter groups whose even or odd
length profiles are not unimodal:

\begin{enumerate}
\item The class corresponding to $A_1^2$ in $W(E_6)$.
\item The classes corresponding to $A_1^2$ and $D_6$ in $W(E_8)$.
\item The classes corresponding to $A_1^2$ and $B_2$ in $W(F_4)$.
\item The class corresponding to $A_1^2$ in $W(H_4)$.
\item The conjugacy classes of $(\mi 1)(\mi 2)(\+ 3 \+ 4)$ and of
$(\mi 1)(\mi 2)(\mi 3)(\mi 4)(\+ 5 \+ 6)$ in $W(D_8)$.
\end{enumerate}

We can therefore immediately say that not all conjugacy classes of
involutions in finite irreducible Coxeter have unimodal length
profiles, and that the even involution length profile of a Coxeter
group is not always unimodal. In particular, Conjecture
\ref{duck}(b) is false. However, on current data, we can make the
following conjectures.

\begin{conj} \begin{enumerate}\item[(i)]
If $X$ is a conjugacy class of involutions in $W(A_n)$ or $W(B_n)$,
then the even/odd length profile of $X$ is unimodal.

\item[(ii)] If $X$ is the set of involutions of odd length in a finite
Coxeter group, then the odd length profile of $X$ is unimodal.
\end{enumerate}
\end{conj}

We end with a small step (extending results of \cite{dukes}) towards
addressing these questions.

\begin{thm}\label{logc}

Let $n \geq 2$ be even. Let $W$ be of type $A_{n-1}, B_n$ or $D_n$,
and let $X$ be the set of involutions in $W$ with no 1-cycles. Then
\begin{align*} L_{W(A_{n-1}),X}(t) &= t^{n/2}\prod_{k=1}^{n/2} \frac{t^{4k-2} - 1}{t^2-1};\\
L_{W(B_n),X}(t) &= t^{n/2}\prod_{k=1}^{n/2} \frac{t^{8k-4}-1}{t^2-1};\\
L_{W(D_n),X}(t) &= t^{n/2}\prod_{k=1}^{n/2}
\frac{(1+t^{4k-4})(t^{4k-2} -1)}{t^2-1}.\end{align*} Hence the
sequences of odd-power and even-power coefficients of $L_{W,X}(t)$
are symmetric, unimodal and, in the case of $W(A_{n-1})$ and
$W(B_{n})$, log-concave.\end{thm}

The result for $L_{W(A_n),X}(t)$ is the second part of Proposition
2.8 and Theorem 2.10 of \cite{dukes}. The result for
$L_{W(D_n),X}(t)$ is Theorems 3.3 and 3.4 of \cite{dukes}. The
following proof is simply an extension of the arguments given there
to include the case of $W(B_n)$, though our derivation of the
formula for $W(D_n)$ is somewhat shorter.

\begin{proof}
Note that $X$ is just the conjugacy class (or union of two classes
for $W(D_n)$) where $e = 0$ and $n  = 2m$. Exactly one of the
sequences will consist entirely of zeros and so will be trivially
log-concave. The expressions for $L_{W,X}(t)$  follow from Theorems
\ref{sn}, \ref{bn} and \ref{dn}, because of the fact that
$L_{W,X}(t) = 0$ whenever $e < 0$ or $2m < n$ leaves just one term
in the recursion. So we have that
\begin{align*} L_{W(A_{n-1}),X}(t) &= \frac{t(t^{2n-2} - 1)}{t^2-1}L_{W(A_{n-3}),X}(t);\\
L_{W(B_n),X}(t) &= \frac{t(t^{4n-4}-1)}{t^2-1}L_{W(B_{n-2}),X}(t);\\
L_{W(D_n),X}(t)
&=\frac{t(1+t^{2n-4})(t^{2n-2}-1)}{t^2-1}L_{W(D_{n-2}),X}(t).\end{align*}
The product expression for each polynomial now follows by induction,
noting that for the base case $n=2$ we have $L_{W(A_1),X}(t) = t =
\frac{t(t^{4\times 1 - 2} - 1)}{t^2 -1}$, $L_{W(B_2),X}(t) = t + t^3
= \frac{t(t^{8 \times 1 - 4}-1)}{t^2-1}$ and $L_{W(D_2),X}(t) = 2t =
\frac{t(1 + t^{4\times 1 - 4})(t^{4\times 1 - 2}-1)}{t^2 - 1}$.\\

 It remains to show that the sequences of coefficients of these
polynomials are symmetric, unimodal and, for types $A_{n-1}$ and
$B_n$, log-concave. Write $s = t^2$. Then
$$t^{-n/2}L_{W,X}(t) = \prod_{k=1}^{n/2} Q_k(s)$$ where $Q_k(s)$ is
either (for $W(A_{n-1})$)
$$\frac{s^{2k-1} - 1}{s-1} = 1 + s + \cdots + s^{2k-2}$$
or (for $W(B_n)$)
$$\frac{s^{4k-2}-1}{s-1} = 1 + s + \cdots + s^{4k-3}$$
or (for $W(D_n)$) $$\frac{(1+s^{2k-2})(s^{2k-1} -1)}{s-1} = 1 + s +
\cdots + s^{2k-3} + 2s^{2k-2} + s^{2k-1} + \cdots + s^{4k-3}.$$ For
all three of these, the sequence of coefficients of $Q_k(s)$ is
symmetric and unimodal with non-negative coefficients. Proposition 1
of \cite{stanley} states that the product of any such polynomials is
also symmetric and unimodal with non-negative coefficients. For the
first two of these cases, corresponding to types $A_{n-1}$ and
$B_n$, the sequence of coefficients of $Q_k(s)$ is always a
non-negative log-concave sequence with no internal zero
coefficients. Proposition 2 of \cite{stanley} states that the
product of any such polynomials is also log-concave with no internal
zero coefficients. Theorem \ref{logc} now follows.
\end{proof}\\

Note that $L_{W(D_n),X}(t)$ is not log-concave in general. For
example, the length profile for $X$, the set of involutions with no
1-cycles in $W(D_4)$, is $[2,2,4,2,2]$, and $2^2 < 2 \times 4$. So
the sequence is not log-concave.

\end{document}